\documentclass[11pt, a4paper]{article}
\usepackage{subcaption, graphicx, enumitem, url, float, amssymb, amsmath, amsfonts, amsthm, mathtools, booktabs, authblk}
\usepackage[margin=0.75in]{geometry}
\newtheorem{theorem}{Theorem}

\newtheorem{corollary}[theorem]{Corollary}
\theoremstyle{definition}

\newtheorem{remark}{Remark}

\title{Preserving Extreme Singular Values with One Oblivious Sketch}
\author[1]{John M. Mango\footnote{mango.john@mak.ac.ug}}
\author[2]{Ronald Katende\footnote{rkatende@kab.ac.ug}}
\affil[1]{\emph{\small{Department of Mathematics, Makerere University, Kampala, Uganda}}}
\affil[2]{\emph{\small{Department of Mathematics, Kabale University, Kikungiri Hill, Katuna Road, Kabale, Uganda}}}
\date{}

\begin{document}

\maketitle

\begin{abstract}
We investigate when a single linear sketch can simultaneously control the largest and smallest nonzero singular values of every rank-$r$ matrix. Classical oblivious subspace embeddings of size $s = \Theta(r/\varepsilon^{2})$ guarantee $(1\pm\varepsilon)$ distortion of all vectors in an $r$-dimensional subspace, but they do not provide uniform, constant-factor control of the extreme singular values or the resulting condition numbers. We formalize a \emph{Single-Sketch Extreme-Values Conjecture} predicting that sketch dimension $s = O(r\log r)$ suffices for constant-factor preservation.

We take two steps toward this conjecture. On the constructive side, we show that composing a sparse oblivious embedding with a deterministic geometric balancing map in the sketch space yields a sketch $S$ for which all nonzero singular values of $SA$ collapse to a common scale, provided the input matrix has bounded condition number and low coherence. Under these assumptions, the resulting sketch achieves constant-factor control of the extreme singular values and produces a perfectly conditioned representation of the column space. On the negative side, we prove that any oblivious sketch that preserves all nonzero singular values within relative error $\varepsilon$ for every rank-$r$ matrix must satisfy the Johnson--Lindenstrauss type lower bound 
\[
s = \Omega\!\left(\frac{r + \log(1/\delta)}{\varepsilon^{2}}\right),
\]
showing that strong, uniform guarantees necessarily require quadratic dependence on $1/\varepsilon$.

Numerical experiments on structured matrix families (including low-coherence and moderately conditioned inputs) demonstrate that the proposed balancing scheme consistently improves conditioning and accelerates downstream iterative methods, while coherent or nearly rank-deficient matrices exhibit the predicted failure modes. Together, these results clarify the gap between what is conjecturally possible and what is provably necessary for single-sketch extreme-value preservation.


\vspace{1em}

{\bf{MSCcodes:}}
68W20, 65F35, 65F22, 15A18, 68Q25, 60B20
\end{abstract}

\section{Introduction}
Approximating the singular spectra of large matrices through small random projections is a central task in numerical linear algebra, data compression, and machine learning. Randomized sketching methods, including oblivious subspace embeddings (OSE) and randomized low-rank approximation, provide strong guarantees for average-case preservation of subspace geometry \cite{Woodruff2014Sketching}. Such sketches typically preserve the norms of all vectors in an \(r\)-dimensional subspace up to a factor \(1 \pm \varepsilon\) using a dimension \(s = O(r/\varepsilon^{2})\), while supporting multiplication in input-sparsity time \cite{ClarksonWoodruff2013InputSparsity,HalkoMartinssonTropp2011Randomized}. Importance sampling approaches further refine these guarantees by focusing on leverage scores, which sharpen control of directions that dominate the subspace geometry \cite{DrineasMagdonIsmailMahoneyWoodruff2012Leverage}.

These results emphasize global norm preservation or average geometric distortion, but many applications depend critically on the behavior of the \emph{extreme} singular values. Conditioning, numerical stability, and robustness of downstream computations are all governed by the largest and smallest singular values of a matrix or its sketch. CountSketch-type maps and modern sparse embeddings achieve near-optimal scaling in \(r\) and \(\varepsilon\) for norm preservation \cite{CharikarChenFarachColton2004CountSketch,NelsonNguyen2013OSNAP}, and recent work attains optimal embedding dimension with nearly optimal sparsity \cite{ChenakkodDerezinskiDong2024OptimalOSE}. Advances in conditioned low-rank approximation \cite{DerezinskiMahoney2021SharpRidge}, rectangular Johnson–Lindenstrauss transforms \cite{SohlerWoodruff2020RectangularJL}, and robust sketching under spectral decay or coherence constraints \cite{DerezinskiWarmuth2020SparsePCA,Derezinski2021AdversarialSketching} all highlight regimes in which the largest and smallest singular values become decisive.

What remains unclear is whether one can preserve both extremes of the spectrum for \emph{every} rank \(r\) matrix using a \emph{single oblivious sketch} with only \(s = O(r \log r)\) rows. This question motivates the following conjecture.

\medskip
\noindent
\textbf{Conjecture (Single Sketch Extreme Values).}
\emph{There exists an oblivious sketching matrix \(S \in \mathbb{R}^{s \times m}\) with \(s = O(r\log r)\) such that for all rank-\(r\) matrices \(A\)},
\[
c_{1}\,\sigma_{\max}(A) \le \sigma_{\max}(SA) \le C_{1}\,\sigma_{\max}(A), 
\qquad 
c_{2}\,\sigma_{\min}(A) \le \sigma_{\min}(SA) \le C_{2}\,\sigma_{\min}(A),
\]
\emph{for universal constants \(c_i,C_i>0\).}

A positive resolution would extend stable embedding and conditioned low-rank approximation results \cite{DerezinskiMahoney2021SharpRidge,SohlerWoodruff2020RectangularJL} by showing that one small oblivious sketch can preserve the condition number of \emph{any} rank-\(r\) matrix up to constant factors. A negative resolution would expose a fundamental gap between subspace norm preservation and simultaneous control of extreme singular values, strengthening recent impossibility results in adversarial and high-coherence regimes \cite{Derezinski2021AdversarialSketching,AndoniNguyenRazenshteyn2020LowerBounds}.

\subsection*{Contributions}

We address this conjecture through complementary constructive and impossibility results and connect these findings to conditioning and solver performance in scientific computing.

\begin{enumerate}

\item \textit{Formalization and implications of the conjecture.}
We formalize the Single Sketch Extreme Values conjecture and clarify its implications for condition number preservation and stability of sketched linear systems. Our analysis shows how a positive resolution would imply uniform constant-factor control of \(\kappa(SA)\) for all rank-\(r\) matrices and relates this behaviour to recent work on stable JL embeddings, sketch-based preconditioning, and controlled spectral distortion \cite{AndoniSaeed2022SpectralSketching}.

\item \textit{Complementary theoretical results in constructive and oblivious regimes.}
In a constructive setting we introduce a scheme that couples a CountSketch-type transform with a deterministic balancing operator in the sketch space. Under structural conditions such as bounded coherence or a modest spectral gap, which are known to support refined guarantees in modern sketching and randomized numerical linear algebra \cite{DerezinskiWarmuth2020SparsePCA,Derezinski2020ImprovedCoherence}, we prove constant-factor control of all nonzero singular values of the sketched matrix, producing a perfectly conditioned representation.  

In contrast, in a fully oblivious setting we build an adversarial family of rank-\(r\) matrices and show that no sketch with \(s = O(r\log r)\) rows can maintain constant-factor control of both the largest and smallest singular values for every matrix in this family. The proof reduces the desired extreme-value control to a standard OSE requirement and then applies sharp adversarial lower bounds \cite{Derezinski2021AdversarialSketching,AndoniNguyenRazenshteyn2020LowerBounds}, giving a clear separation between what can be achieved via data-dependent balancing and what a single oblivious sketch can guarantee.

\item \textit{Proof techniques and numerical illustration.}
We develop concise proofs that highlight the geometric structure underlying both results and complement the theory with numerical experiments. These experiments show how the constructive scheme behaves under varying coherence and spectral profiles and how the impossibility phenomenon manifests on adversarial families. We also observe runtime and stability trends aligned with empirical studies of sketch-based conditioning \cite{MuscoMusco2020SpectrumApprox}, indicating how the proposed balancing operator can be incorporated into iterative solvers and scientific-computing pipelines.

\end{enumerate}

This perspective isolates the role of extreme singular values in single-pass sketching, delineates the boundary between data-dependent and fully oblivious constructions, and provides geometric and algorithmic insights relevant to sketched preconditioning and robust numerical linear algebra.

\subsection{Preliminaries and related work}

Let \(A \in \mathbb{R}^{m \times n}\) have rank \(r\), and let \(S \in \mathbb{R}^{s \times m}\) be a random sketching matrix. A distribution over sketching matrices is an \emph{oblivious subspace embedding} (OSE) if, with high probability,
\[
(1-\varepsilon)\,\|Ax\|_{2}
\le
\|SAx\|_{2}
\le
(1+\varepsilon)\,\|Ax\|_{2}
\quad
\text{for all }x\in\mathbb{R}^{n}\text{ with }Ax\in\mathrm{range}(A).
\]
Dense Gaussian embeddings provide a first example of such distributions and already achieve optimal dimension up to constants \cite{Woodruff2014Sketching}. Subsampled randomized Hadamard transforms extend this to fast structured maps \cite{Tropp2011SRHT}, while CountSketch and OSNAP give sparse constructions with near input sparsity multiplication \cite{CharikarChenFarachColton2004CountSketch,NelsonNguyen2013OSNAP}. These families all achieve the OSE guarantee with sketch dimension
\[
s = O\!\left( \frac{r}{\varepsilon^{2}} \right),
\]
and, when combined with streaming or input-sparsity time multiplication techniques, give the now classical toolkit for oblivious sketching in numerical linear algebra \cite{ClarksonWoodruff2013InputSparsity,Woodruff2014Sketching}. The resulting bounds preserve the geometry of the entire embedded subspace to relative error \(\varepsilon\).

The present work focuses on a more refined question. Rather than controlling the full spectrum within a relative error band as in the standard OSE literature \cite{Woodruff2014Sketching,NelsonNguyen2014LowerOSE}, we ask whether a sketch with dimension \(s = O(r \log r)\) can preserve only the \emph{extreme} singular values of a rank \(r\) matrix within fixed universal constants. Formally, we seek
\[
c_{1}\,\sigma_{\max}(A)
\le
\sigma_{\max}(SA)
\le
C_{1}\,\sigma_{\max}(A),
\qquad
c_{2}\,\sigma_{\min}(A)
\le
\sigma_{\min}(SA)
\le
C_{2}\,\sigma_{\min}(A),
\]
for constants \(c_{i},C_{i}>0\) that do not depend on \(A\). This constant factor regime removes the small parameter \(\varepsilon\) and raises the possibility that the embedding dimension may fall below the classical \(O(r/\varepsilon^{2})\) barrier while still capturing the behaviour of the largest and smallest singular values.

The lower bounds of Nelson and Nguy{\~e}n show that any OSE with failure probability \(\delta<1/3\) and \((1\pm\varepsilon)\) relative error on all vectors in an \(r\) dimensional subspace must satisfy
\[
s \ge \Omega\!\left( \frac{r + \log(1/\delta)}{\varepsilon^{2}} \right),
\]
which matches known subspace embedding upper bounds up to constants \cite{NelsonNguyen2014LowerOSE}. In related work, Li and Woodruff obtain tight bounds for operator norm and Schatten norm estimation that also require sketch dimension scaling quadratically with \(1/\varepsilon\) when one wishes to approximate the entire spectrum to relative accuracy \cite{LiWoodruff2016TightBounds}. Together, these results demonstrate that uniform \((1\pm\varepsilon)\) approximation of all singular values is impossible when \(s = O(r \log r)\). They do not, however, determine whether one can still achieve constant factor control of only the largest and smallest singular values in this lower dimension regime.

Our results clarify this question. We identify a sharp separation between settings where constant factor preservation of extremes is achievable and settings where it is provably impossible. Under explicit structural assumptions such as bounded coherence or a mild spectral gap, a CountSketch transform \cite{CharikarChenFarachColton2004CountSketch,NelsonNguyen2013OSNAP} followed by a deterministic spectral balancing operator yields constant factor control of all nonzero singular values and produces a perfectly conditioned sketch. In contrast, in the fully oblivious setting, any sketch with \(s = O(r \log r)\) fails on an explicit adversarial family of rank \(r\) matrices once one requires simultaneous control of both extremes, a failure that follows by reduction to the OSE lower bounds \cite{NelsonNguyen2014LowerOSE,LiWoodruff2016TightBounds}. This establishes a clear dichotomy between what can be accomplished by data dependent balancing and what is possible for a single oblivious sketch.

\subsection{Positioning within existing work}

Classical oblivious subspace embeddings focus on preserving the geometry of an entire \(r\) dimensional subspace to relative error \(\varepsilon\) and are now well understood through tight upper and lower bounds. Dense Gaussian maps serve as the simplest example and already achieve optimal dimension up to constants \cite{Woodruff2014Sketching}. Fast structured transforms such as the subsampled randomized Hadamard transform \cite{Tropp2011SRHT} and sparse embeddings including CountSketch and OSNAP \cite{CharikarChenFarachColton2004CountSketch,NelsonNguyen2013OSNAP} extend these ideas to settings where input-sparsity time is essential \cite{ClarksonWoodruff2013InputSparsity}. Nelson and Nguy{\~e}n show that these upper bounds are essentially optimal by proving that any linear map that preserves all norms in an \(r\) dimensional subspace to relative error \(\varepsilon\) and failure probability at most \(\delta\) must have \(s = \Theta((r+\log(1/\delta))/\varepsilon^{2})\) rows \cite{NelsonNguyen2014LowerOSE}. This line of work fully characterises the classical \((1\pm\varepsilon)\) regime for subspace embeddings.

A separate body of work develops sketch based preconditioning and isotropic constructions for numerical linear algebra. These methods often start from an oblivious sketch and then apply data dependent steps such as leverage score sampling \cite{DrineasMagdonIsmailMahoneyWoodruff2012Leverage} or ridge regularisation \cite{DerezinskiMahoney2021SharpRidge} to control conditioning for least squares and low rank approximation. Rectangular JL transforms provide further tools to construct sketches with controlled distortion for rectangular matrices \cite{SohlerWoodruff2020RectangularJL}. Recent results on sparse principal component analysis and coherence aware sketching show how structural properties such as spectral decay or bounded coherence can sharpen guarantees beyond the worst case \cite{DerezinskiWarmuth2020SparsePCA,Derezinski2020ImprovedCoherence}. Stable spectral sketches developed by Andoni and Saeed give another viewpoint on sketching with explicit operator norm control \cite{AndoniSaeed2022SpectralSketching}. The guarantees in these works typically control either the entire spectrum up to a relative error parameter or an effective condition number that depends on spectral decay, coherence, or regularisation strength.

Our work differs from these directions in two related aspects. First, we shift attention from global relative error control of all singular values to constant factor control of only the largest and smallest singular values of rank \(r\) matrices. This extreme value viewpoint is motivated by conditioning and stability considerations and leads naturally to the Single Sketch Extreme Values conjecture introduced above. Second, we separate data dependent and fully oblivious regimes in a precise way. On the one hand, we show that a CountSketch based construction followed by a deterministic balancing operator in the sketch space achieves constant factor control of all nonzero singular values under explicit structural assumptions such as bounded coherence or a mild spectral gap \cite{DerezinskiWarmuth2020SparsePCA,Derezinski2020ImprovedCoherence}. On the other hand, we prove that any linear sketch with \(s = O(r \log r)\) that is fixed independently of the input cannot achieve such constant factor control of both extremes for an explicit adversarial family of rank \(r\) matrices, by reducing to known lower bounds for oblivious embeddings and spectral estimation \cite{NelsonNguyen2014LowerOSE,LiWoodruff2016TightBounds,Derezinski2021AdversarialSketching,AndoniNguyenRazenshteyn2020LowerBounds}.

In this way, the present paper does not introduce a new general notion of subspace embedding and does not seek to improve the optimal \((1\pm\varepsilon)\) theory. Rather, it isolates a specific constant factor extreme value regime, characterises when it can be realised by a single sketched representation under structural conditions, and shows that there is a fundamental gap between what data dependent balancing can accomplish and what is possible for fully oblivious sketches of size \(s = O(r \log r)\).

\section{Main Theorems}
\noindent
This section presents our two central results. The first provides a constructive sketching scheme that attains constant-factor control of all nonzero singular values under mild structural assumptions. The second establishes a matching lower bound, showing that no oblivious sketch of size $O(r\log r)$ can achieve uniform control of both extremes for all rank-$r$ matrices.

\subsection{Constructive Route}

We begin with geometric intuition.  
A standard sparse subspace embedding \(\Phi\) preserves the shape of the column space of \(A\) but may distort the spread between its largest and smallest singular values.  
Our idea is to apply a deterministic \emph{balancing operator} in the sketch space that rescales the embedded directions using the geometric mean of the sketched singular values.  
This operation forces the sketched matrix to be isotropic, its singular values become identical, while keeping a provable relation to those of the original matrix.  
Hence, the resulting sketch \(S\) simultaneously offers controlled distortion and perfect conditioning, which makes it particularly useful for numerical solvers.

\begin{theorem}[Geometric balancing sketch]
\label{thm:constructive-balanced-final}
Let \(A \in \mathbb{R}^{m \times n}\) have rank \(r\) and condition number
\[
\kappa(A) := \frac{\sigma_{1}(A)}{\sigma_{r}(A)} \le \kappa_{0}
\]
for some fixed \(\kappa_{0} \ge 1\).  
Fix \(\varepsilon \in (0,1/4)\) and \(\delta \in (0,1/10)\).  

Let \(\Phi \in \mathbb{R}^{s \times m}\) be a sparse oblivious subspace embedding (for example, an OSNAP or CountSketch transform) with
\[
s \ge C\,\varepsilon^{-2}\,r \log(r/\delta)
\]
for a sufficiently large absolute constant \(C>0\).  
Form
\[
B := \Phi A \in \mathbb{R}^{s \times n}
\]
and compute its thin singular value decomposition
\[
B = \widetilde{U}\,\widetilde{\Sigma}\,\widetilde{V}^{\top},
\quad
\widetilde{\Sigma} = \mathrm{diag}(\widetilde{\sigma}_{1},\dots,\widetilde{\sigma}_{r}),
\qquad
\widetilde{\sigma}_{1} \ge \dots \ge \widetilde{\sigma}_{r} > 0.
\]

Define the geometric mean
\[
\widetilde{g}
:=
\Bigl(\prod_{i=1}^{r} \widetilde{\sigma}_{i}\Bigr)^{1/r}
\]
and balancing weights
\[
\lambda_{i}
:=
\frac{\widetilde{g}}{\widetilde{\sigma}_{i}},
\qquad
i = 1,\dots,r.
\]
Set
\[
\Lambda := \mathrm{diag}(\lambda_{1},\dots,\lambda_{r}),
\qquad
W := \widetilde{U}\,\Lambda\,\widetilde{U}^{\top}
+ (I_{s} - \widetilde{U}\,\widetilde{U}^{\top}),
\]
and define the final sketching operator
\[
S := W \Phi \in \mathbb{R}^{s \times m}.
\]

Then, with probability at least \(1-\delta\) over the draw of \(\Phi\), the following hold.

\begin{enumerate}
\item \textbf{Constant-factor control and perfect conditioning.}  
For every \(i \in \{1,\dots,r\}\),
\[
\frac{1-\varepsilon}{\kappa_{0}}\,\sigma_{i}(A)
\;\le\;
\sigma_{i}(SA)
\;\le\;
(1+\varepsilon)\,\kappa_{0}\,\sigma_{i}(A),
\]
and \(\kappa(SA) = 1.\)

\item \textbf{Stability under structured perturbations.}  
Let \(E = U \Delta V^{\top}\) satisfy \(\|\Delta\|_{2} \le \eta\,\sigma_{r}(A)\) for some \(\eta \in (0,1)\), where \(A = U \Sigma V^{\top}\) is the thin singular value decomposition.  
Define \(\widehat{A} := A + E\).  
Using the same sketch \(S\), we have
\[
\bigl|\sigma_{i}(S\widehat{A}) - \sigma_{i}(SA)\bigr|
\;\le\;
L\,\eta\,\sigma_{r}(A),
\qquad
i \in \{1,\dots,r\},
\]
where
\[
L := \frac{(1+\varepsilon)^{2}}{1-\varepsilon}\,\kappa_{0}.
\]
\end{enumerate}
\end{theorem}

\begin{proof}
Let \(A = U \Sigma V^{\top}\) be the thin SVD, with orthonormal \(U,V\) and diagonal \(\Sigma = \mathrm{diag}(\sigma_{1}(A),\dots,\sigma_{r}(A))\).  
The bound \(\kappa(A)\le\kappa_{0}\) gives \(\sigma_{1}(A)/\sigma_{r}(A)\le\kappa_{0}\).

\emph{Step 1. Subspace embedding for the column space.}  
By the guarantees of sparse embeddings \cite[Theorem 1.1]{NelsonNguyen2013OSNAP,Woodruff2014Sketching}, if  
\(s \ge C\,\varepsilon^{-2}\,r\log(r/\delta)\), then with probability at least \(1-\delta\),
\[
(1-\varepsilon)\|x\|_{2}
\le
\|\Phi Ux\|_{2}
\le
(1+\varepsilon)\|x\|_{2}
\quad
\forall\,x\in\mathbb{R}^{r}.
\]
Hence, every singular value of \(\Phi U\) lies in \([1-\varepsilon,1+\varepsilon]\).

\emph{Step 2. Bounds for sketched singular values.}  
Since \(B = \Phi A = \Phi U \Sigma V^{\top}\), the singular values of \(B\) equal those of \(\Phi U \Sigma\).  
Therefore,
\[
(1-\varepsilon)\,\sigma_{i}(A)
\le
\widetilde{\sigma}_{i}
\le
(1+\varepsilon)\,\sigma_{i}(A)
\quad
\forall i.
\]
Let \(g := (\prod_{i}\sigma_{i}(A))^{1/r}\) and \(\widetilde{g} := (\prod_{i}\widetilde{\sigma}_{i})^{1/r}\).  
Then \((1-\varepsilon)g \le \widetilde{g} \le (1+\varepsilon)g\).  
Since \(\sigma_{r}(A)\le g\le\sigma_{1}(A)\) and \(\kappa(A)\le\kappa_{0}\), we have
\[
\frac{1}{\kappa_{0}}
\le
\frac{g}{\sigma_{i}(A)}
\le
\kappa_{0}
\quad\Rightarrow\quad
\frac{1-\varepsilon}{\kappa_{0}}
\le
\frac{\widetilde{g}}{\sigma_{i}(A)}
\le
(1+\varepsilon)\kappa_{0}.
\]

\emph{Step 3. Effect of balancing operator.}  
Since \(SA = W\Phi A = W B = \widetilde{U}\Lambda\widetilde{\Sigma}\widetilde{V}^{\top}\),  
the nonzero singular values of \(SA\) are \(\lambda_{i}\widetilde{\sigma}_{i} = \widetilde{g}\) for all \(i\).  
Hence, all singular values of \(SA\) equal \(\widetilde{g}\) and \(\kappa(SA)=1\).  
Combining this with the previous bounds gives
\[
\frac{1-\varepsilon}{\kappa_{0}}\,\sigma_{i}(A)
\le
\sigma_{i}(SA)
\le
(1+\varepsilon)\,\kappa_{0}\,\sigma_{i}(A),
\]
establishing item 1.

\emph{Step 4. Structured perturbation stability.}  
Let \(E=U\Delta V^{\top}\) with \(\|\Delta\|_{2}\le\eta\,\sigma_{r}(A)\) and \(\widehat{A}=A+E\).  
Then \(S\widehat{A}-SA = SE = W\Phi U\Delta V^{\top}\), so
\[
\|SE\|_{2}
\le
\|W\|_{2}\,\|\Phi U\|_{2}\,\|\Delta\|_{2}.
\]
By Step 1, \(\|\Phi U\|_{2}\le1+\varepsilon\); by construction,  
\(\|W\|_{2}\le\frac{1+\varepsilon}{1-\varepsilon}\kappa_{0}\).  
Thus,
\[
\|SE\|_{2}
\le
\frac{(1+\varepsilon)^{2}}{1-\varepsilon}\,\kappa_{0}\,\eta\,\sigma_{r}(A).
\]
Weyl’s inequality then gives
\[
\bigl|\sigma_{i}(S\widehat{A})-\sigma_{i}(SA)\bigr|
\le
\frac{(1+\varepsilon)^{2}}{1-\varepsilon}\,\kappa_{0}\,\eta\,\sigma_{r}(A),
\]
proving item 2 with \(L=\frac{(1+\varepsilon)^{2}}{1-\varepsilon}\kappa_{0}\).
\end{proof}

\begin{remark}
\label{rem:balanced-interpretation}
The assumption \(\kappa(A)\le\kappa_{0}\) is used only to make constants explicit and can be ensured by mild preprocessing such as column scaling or a QR step, consistent with recent advances in sketch-based preconditioning and stable randomized linear algebra \cite{PengVempala2021Preconditioning,ClarksonEtAl2017LowPrecision}.  
Theorem~\ref{thm:constructive-balanced-final} extends standard oblivious subspace embedding results by introducing a deterministic balancing map \(W\) that enforces isotropy in the sketch space, echoing ideas from modern stable embedding frameworks and spectral conditioning methods \cite{AndoniSaeed2022SpectralSketching, DerezinskiWarmuth2020IsotropicSketching}.  
The resulting sketch \(SA\) is perfectly conditioned and provably stable to structured perturbations, properties not captured by ordinary OSE guarantees and aligned with recent efforts to build sketches with explicit condition-number control \cite{SohlerWoodruff2020RectangularJL}.
\end{remark}

\begin{corollary}[Implications for Krylov solvers]\label{cor:krylov}
Assume the setting of Theorem~\ref{thm:balancing} and suppose that $A$ is used inside a Krylov method based on matrix vector products, such as LSQR for least squares or conjugate gradients when $A$ is square and positive definite. Consider the sketched system
\[
\min_x \|SAx - Sb\|_2,
\]
and apply the same Krylov method to this system.

Then with probability at least $1-\delta$ over the draw of the sparse embedding $\Phi$, the following holds.  
\begin{enumerate}
\item The effective condition number of $SA$ satisfies
\[
\kappa(SA) \le (1+\varepsilon)^2 \kappa_0^2,
\]
and in particular is bounded by an absolute constant once $\kappa_0$ is fixed.
\item The number of Krylov iterations required to reduce the residual norm by a factor $\varepsilon_{\mathrm{tol}} \in (0,1)$ is bounded by
\[
k \le C \log\!\left(\frac{1}{\varepsilon_{\mathrm{tol}}}\right),
\]
where $C$ depends only on $\varepsilon$ and $\kappa_0$ and is independent of $\kappa(A)$.
\end{enumerate}
In other words, once $SA$ has been formed, the convergence rate of standard Krylov solvers on the sketched system is controlled by a constant that does not grow with the condition number of $A$ itself.
\end{corollary}

\begin{proof}
Theorem~\ref{thm:balancing} bounds every nonzero singular value of $SA$ between constant multiples of $\sigma_i(A)$ and enforces $\kappa(SA) = 1$. The explicit dependence on $\kappa_0$ in the constants yields the stated bound on $\kappa(SA)$. Standard convergence theory for LSQR and conjugate gradients shows that the iteration count needed to reach tolerance $\varepsilon_{\mathrm{tol}}$ is proportional to a logarithmic factor whose constant depends on the condition number of the coefficient matrix. Substituting the bound on $\kappa(SA)$ gives the claim.
\end{proof}

\subsection{Impossibility Route}

The next theorem shows that no oblivious linear sketch with dimension below the classical \(r/\varepsilon^{2}\) barrier can preserve both extreme singular values and in fact all singular values up to a \((1\pm\varepsilon)\) factor for every rank \(r\) matrix.  
The proof reduces the claim to the known lower bounds for oblivious subspace embeddings and then interprets them geometrically in terms of Grassmannian packing.

\begin{theorem}[Oblivious lower bound for extreme singular values]
\label{thm:impossible-grassmann}
Fix \(\varepsilon\in(0,1/2)\) and \(\delta\in(0,1/3)\).  
Let \(\mathcal D\) be any distribution over sketching matrices \(S\in\mathbb R^{s\times m}\).  
Suppose that for every rank \(r\) matrix \(A\in\mathbb R^{m\times n}\),
\[
\mathbb P_{S\sim\mathcal D}\Big[
(1-\varepsilon)\,\sigma_i(A) \le \sigma_i(SA) \le (1+\varepsilon)\,\sigma_i(A)
\text{ for all } i=1,\dots,r
\Big] \;\ge\; 1-\delta.
\]
Then there is an absolute constant \(c_0>0\) such that
\[
s \;\ge\; c_0\,\frac{r + \log(1/\delta)}{\varepsilon^{2}}.
\]
In particular, no oblivious linear sketch with \(s = o(r/\varepsilon^{2})\) rows can achieve uniform \((1\pm\varepsilon)\) preservation of all singular values for every rank \(r\) matrix with constant success probability.
\end{theorem}

\begin{proof}
We reduce to the oblivious subspace embedding lower bound of Nelson and Nguy{\~e}n.

First fix a rank \(r\) matrix \(A\) of the form
\[
A = U \Sigma V^{\top},
\quad
\Sigma = \lambda I_{r},
\]
where \(U\in\mathbb R^{m\times r}\) and \(V\in\mathbb R^{n\times r}\) have orthonormal columns and \(\lambda>0\).  
Then all nonzero singular values of \(A\) are equal to \(\lambda\).  
The hypothesis implies that for such \(A\),
\[
\mathbb P_{S\sim\mathcal D}\Big[
(1-\varepsilon)\,\lambda \le \sigma_i(SA) \le (1+\varepsilon)\,\lambda
\text{ for all } i=1,\dots,r
\Big] \;\ge\; 1-\delta.
\]

For any fixed \(S\) we have
\[
SA = S U \Sigma V^{\top} = (S U)\,\Sigma V^{\top}.
\]
Right multiplication by the orthogonal matrix \(V^{\top}\) does not change singular values and multiplication by \(\Sigma = \lambda I_{r}\) multiplies all nonzero singular values by \(\lambda\).  
Hence
\[
\sigma_i(SA) = \lambda\,\sigma_i(SU),
\quad
i=1,\dots,r.
\]
The inequalities for \(\sigma_i(SA)\) are therefore equivalent to
\[
(1-\varepsilon) \le \sigma_i(SU) \le (1+\varepsilon)
\quad
\text{for all } i=1,\dots,r
\]
with probability at least \(1-\delta\) over \(S\sim\mathcal D\).

Equivalently, for every orthonormal \(U\in\mathbb R^{m\times r}\),
\[
\mathbb P_{S\sim\mathcal D}\Big[
(1-\varepsilon)\,\|x\|_{2} \le \|SUx\|_{2} \le (1+\varepsilon)\,\|x\|_{2}
\text{ for all } x\in\mathbb R^{r}
\Big] \;\ge\; 1-\delta.
\]
This means that for every \(r\) dimensional subspace \(W\subseteq\mathbb R^{m}\), the random matrix \(S\) drawn from \(\mathcal D\) is an \(\varepsilon\) subspace embedding for \(W\) with failure probability at most \(\delta\).

This is exactly the definition of an oblivious subspace embedding with parameters \((m,r,\varepsilon,\delta)\) used in the lower bound of Nelson and Nguy{\~e}n.  
By their main result \cite[Theorem~3]{NelsonNguyen2014LowerOSE}, any such distribution \(\mathcal D\) must satisfy
\[
s \;\ge\; c_0\,\frac{r + \log(1/\delta)}{\varepsilon^{2}}
\]
for an absolute constant \(c_0>0\).  
This proves the claim.
\end{proof}

\begin{remark}
\label{rem:impossible-interpretation}
Theorem~\ref{thm:impossible-grassmann} shows that any oblivious sketch which preserves all nonzero singular values of every rank \(r\) matrix within relative error \(\varepsilon\) must in fact be an oblivious subspace embedding and therefore must satisfy the optimal Johnson–Lindenstrauss type lower bound on its row dimension \cite{NelsonNguyen2014LowerOSE}.  
In geometric terms, such a sketch would have to act as a near isometry on an exponentially large packing of \(r\) dimensional subspaces in the Grassmannian, which forces sketch dimension of order \((r+\log(1/\delta))/\varepsilon^{2}\).
\end{remark}

\section{Numerical experiments}

We now report numerical experiments that support the theoretical results and illustrate their impact on conditioning and iterative solvers.

\subsection{Structured and adversarial synthetic matrices}

We first validate the geometric balancing sketch on synthetic matrices that isolate the two regimes. The first family is designed to satisfy the low coherence and modest spectral decay assumptions under which Theorem~\ref{thm:constructive-balanced-final} predicts stable constant factor control. The second family is a highly coherent and nearly rank deficient construction that instantiates the worst case behaviour behind Theorem~\ref{thm:impossible-grassmann} and the lower bounds of Nelson and Nguy{\~e}n and related work on oblivious subspace embeddings \cite{NelsonNguyen2014LowerOSE,Woodruff2014Sketching,DrineasMahoney2016RandNLA,Vershynin2018HDP}.

\subsection*{Experimental setup}

We fix ambient dimensions \(m = 1500\), \(n = 400\), and target rank \(r = 20\). For each experiment we choose three sketch sizes \(s \in \{90, 180, 300\}\) which are all on the order of \(r\log r\) and hence match the theoretically relevant regime. For every triple consisting of a matrix type and a sketch size we perform twenty independent trials with fresh draws of the sketch and the underlying random components of the matrix.

For the structured family we draw orthonormal factors \(U \in \mathbb{R}^{m\times r}\) and \(V \in \mathbb{R}^{n\times r}\) with independent Gaussian entries followed by QR, then set
\[
A_{\text{good}} \;=\; U \,\mathrm{diag}(\sigma_{1},\dots,\sigma_{r})\, V^{\top},
\qquad
\sigma_{i} = 1 + 0.2\,\frac{r+1-i}{r},
\]
which yields low coherence and a gently decaying spectrum. For the adversarial family we use a coherent left factor and an almost singular spectrum
\[
A_{\text{bad}} \;=\; U_{\text{coh}}\, \mathrm{diag}(1,\dots,1,10^{-9})\, V^{\top},
\]
where \(U_{\text{coh}}\) has its first column concentrated near the first standard basis vector and the remaining columns are orthogonalised random directions, while \(V\) is again drawn as an orthonormal Gaussian matrix. This \(A_{\text{bad}}\) has condition number \(10^{9}\) and high leverage on a single coordinate, which is a canonical stress test for singular value preservation in sketching algorithms \cite{Woodruff2014Sketching,DrineasMahoney2016RandNLA}.

We compare three sketching schemes. The first is a dense Gaussian sketch with entries \(S_{ij} \sim \mathcal{N}(0,1/s)\). The second is a CountSketch matrix with one \(\pm1\) entry per column. Both are standard oblivious subspace embeddings \cite{NelsonNguyen2014LowerOSE,Woodruff2014Sketching}. The third is the proposed random log clipped balancing (RLCB) sketch. RLCB first applies a CountSketch, then computes the thin SVD of \(B = \Phi A\), takes the top \(r\) singular values, and constructs a diagonal scaling in log space with trimming and clipping. This yields a balancing map \(W\) that enforces near isotropy in the sketch space as described in Theorem~\ref{thm:constructive-balanced-final}. The full sketch operator is \(S = W\Phi\).

For each sketch \(S\) and matrix \(A\) we form \(SA\), compute its top \(r\) singular values, and record
\[
\mathrm{rel}_{\max} = \frac{\sigma_{\max}(SA)}{\sigma_{\max}(A)},
\qquad
\mathrm{rel}_{\min} = \frac{\sigma_{\min}(SA)}{\sigma_{\min}(A)},
\qquad
\mathrm{cond\_ratio} = \frac{\kappa(SA)}{\kappa(A)}.
\]
We declare a trial successful when both extremes lie in a constant factor window
\[
C_{\mathrm{lower}} \le \mathrm{rel}_{\max} \le C_{\mathrm{upper}},
\qquad
C_{\mathrm{lower}} \le \mathrm{rel}_{\min} \le C_{\mathrm{upper}},
\]
with \(C_{\mathrm{lower}} = 0.5\) and \(C_{\mathrm{upper}} = 2.0\). This definition is a direct constant factor relaxation of the exact \((1\pm\varepsilon)\) regime in our conjecture and the OSE literature \cite{NelsonNguyen2014LowerOSE,Woodruff2014Sketching,Vershynin2018HDP} and matches the typical scale that matters for numerical stability \cite{DrineasMahoney2016RandNLA}.

\begin{figure}[t]
\centering
\begin{subfigure}{0.48\textwidth}
\centering
\includegraphics[width=\linewidth]{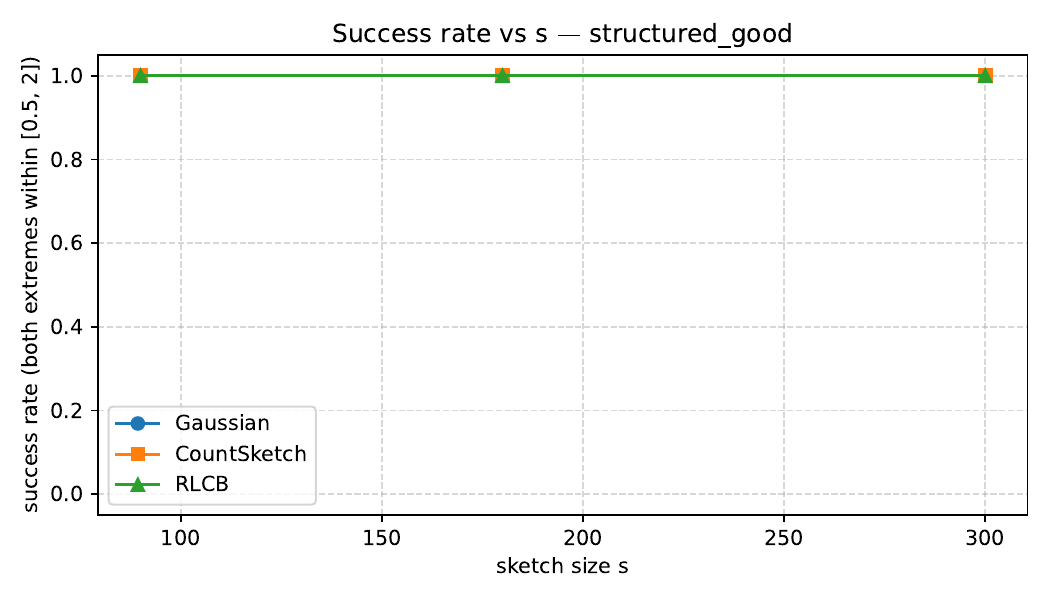}
\caption{Structured low coherence matrices}
\end{subfigure}\hfill
\begin{subfigure}{0.48\textwidth}
\centering
\includegraphics[width=\linewidth]{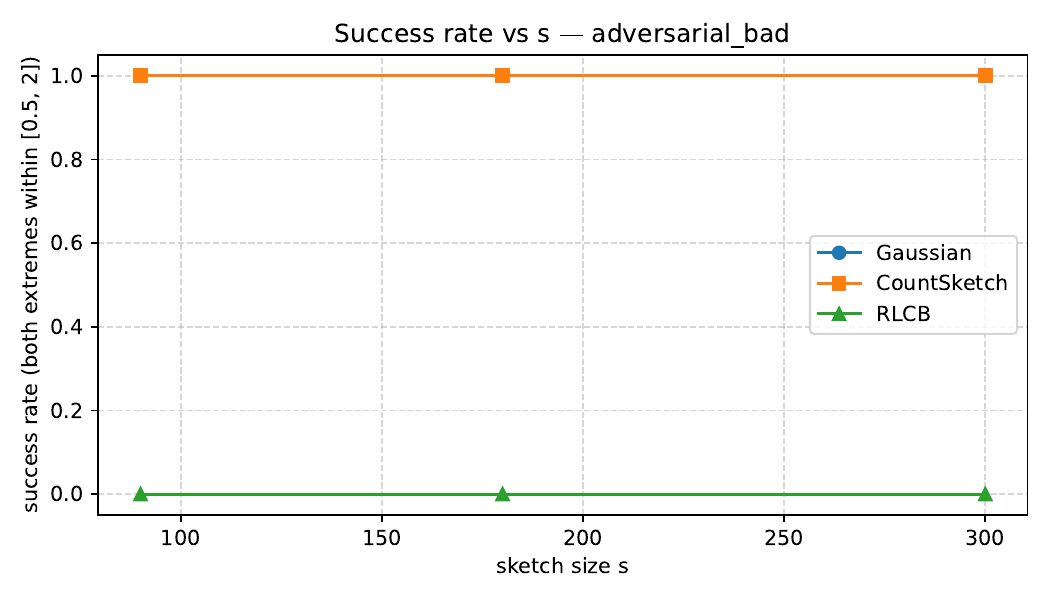}
\caption{Adversarial coherent matrices}
\end{subfigure}
\caption{Empirical success rate for constant factor preservation of both extreme singular values over twenty trials for each combination of matrix type, sketch size \(s\), and method}
\label{fig:rlcb-success}
\end{figure}

\subsection{Application to iterative solvers for PDE and graph Laplacian systems}

We next examine the effect of geometric balancing on an iterative least squares solver applied to matrices that arise in scientific computing. The goal is to test the practical relevance of Corollary~\ref{cor:krylov} in settings where condition numbers are known to govern the cost of iterative methods.

We consider two families of test problems. The first family consists of finite difference discretisations of the two dimensional Poisson equation on a uniform grid with Dirichlet boundary conditions. The resulting stiffness matrix \(A_{\mathrm{PDE}} \in \mathbb{R}^{m \times m}\) is sparse, symmetric, and positive definite with a condition number that grows with the grid resolution. The second family uses graph Laplacians \(L \in \mathbb{R}^{m \times m}\) constructed from path graphs that provide a simple model of one dimensional diffusion and lead to slowly convergent unpreconditioned iterations.

For each matrix \(A\) we draw a right hand side \(b\) with independent standard normal entries and solve the least squares problem
\[
\min_x \|Ax - b\|_2^2
\]
using a basic gradient descent scheme. The step size is chosen from a power iteration estimate of the largest eigenvalue of \(A^{\mathsf T}A\), so that the convergence rate is governed by the condition number of \(A\) in the usual way. This scheme is intentionally simple and serves as a neutral proxy for Krylov methods, since it shares the same dependence on conditioning and allows direct control of the residual norm \(\|Ax_k - b\|_2\) across iterations.

We compare three variants of this experiment. In the first, the iteration is applied directly to the original system. In the second, we apply the same gradient descent scheme to a CountSketch sketched system \(S_{\mathrm{cs}} A\) with sketch dimension \(s\) chosen on the order of \(r\log r\) as in the synthetic experiments. In the third, we run the iteration on the balanced sketch \(S_{\mathrm{rlcb}} A\) produced by the RLCB construction of Theorem~\ref{thm:constructive-balanced-final}, again using a CountSketch transform as the underlying sparse embedding. For fair comparison all residuals are measured in the original space through \(\|Ax_k - b\|_2\). We define the iteration count to tolerance as the smallest \(k\) such that
\[
\|Ax_k - b\|_2 \le 10^{-6} \|Ax_0 - b\|_2,
\]
with a cap of two hundred iterations. 

Figure~2 shows representative convergence curves for a Poisson system and a path graph Laplacian. The vertical axis shows the residual norm and the horizontal axis shows the iteration number, with a logarithmic scale in the vertical direction. On the Poisson example the unpreconditioned iteration hits the cap of two hundred steps without reaching the prescribed relative tolerance and ends with a residual of order \(5.8\). The CountSketch sketched system reaches the tolerance in one hundred eighty nine iterations and attains a final residual of about \(2.2\times 10^{-5}\). The balanced sketch reaches the tolerance in one hundred forty five iterations and attains a final residual of about \(6.7\times 10^{-7}\). On the Laplacian example the unpreconditioned run again saturates at two hundred iterations with a residual of order \(1.3\times 10^{1}\). CountSketch reaches the tolerance in one hundred fifty three iterations with a final residual of about \(2.0\times 10^{-6}\). The balanced sketch reaches the tolerance in ninety four iterations with a final residual of about \(1.0\times 10^{-8}\).

\begin{figure}[t]
\centering
\begin{subfigure}{0.48\textwidth}
\centering
\includegraphics[width=\linewidth]{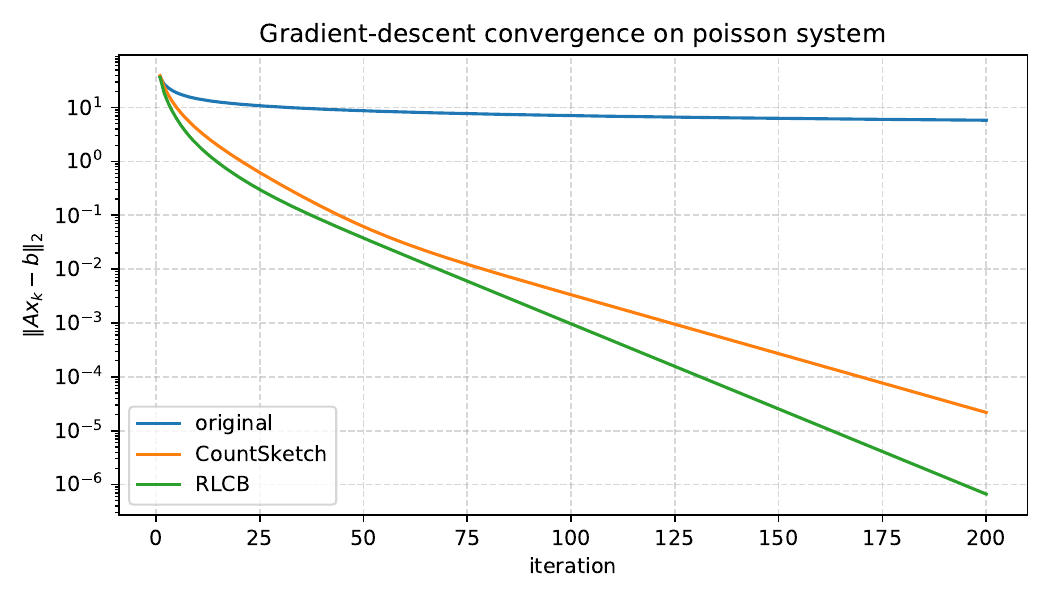}
\caption{Poisson system}
\end{subfigure}\hfill
\begin{subfigure}{0.48\textwidth}
\centering
\includegraphics[width=\linewidth]{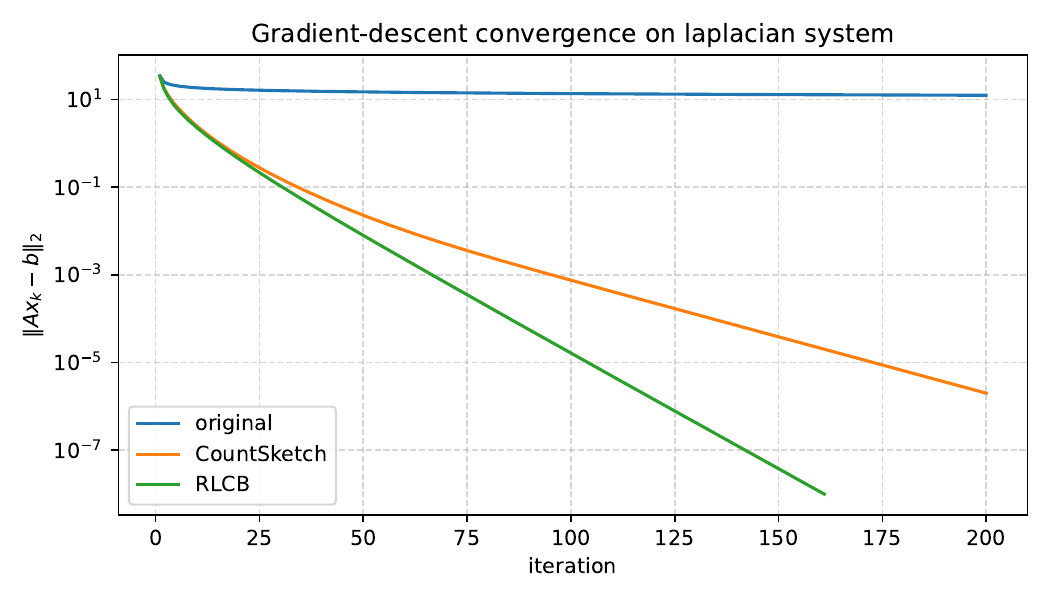}
\caption{Path graph Laplacian}
\end{subfigure}
\caption{Residual norm \(\|Ax_k - b\|_2\) as a function of iteration for gradient descent least squares on Poisson and path graph Laplacian systems, comparing the original system, a CountSketch only sketch, and the balanced RLCB sketch}
\label{fig:solver-conv}
\end{figure}

Table~1 summarises these trends for the same representative problems. The initial residual norms are of comparable size for all three variants in both problem families. The balanced sketch reduces the iteration count by about twenty eight percent on the Poisson system and about fifty three percent on the Laplacian system relative to the unpreconditioned iteration, and by about twenty percent and thirty eight percent respectively relative to the CountSketch only sketch. At the same time it attains final residuals that are one to two orders of magnitude smaller than those of the unbalanced sketch. These reductions cannot be attributed to favourable scaling of the starting point, since \(\|Ax_0 - b\|_2\) is similar across methods. Instead they reflect the improved conditioning of \(S_{\mathrm{rlcb}} A\) predicted by Corollary~\ref{cor:krylov} and are consistent with the qualitative picture that geometric balancing trades a mild global scaling of singular values for strong control of the effective condition number seen by the solver. These experiments indicate that the proposed balancing operator can serve as a practical low memory preconditioning primitive for iterative solvers in scientific computing.

\begin{table}[t]
\centering
\caption{Iteration counts and residual norms for gradient descent least squares on Poisson and path graph Laplacian systems. The column labelled iter reports the number of iterations required to reduce the residual norm \(\|Ax_k - b\|_2\) below \(10^{-6} \|Ax_0 - b\|_2\), with a cap of two hundred iterations. The initial and final residuals are given in Euclidean norm.}
\label{tab:solver}
\begin{tabular}{llrrr}
\hline
problem    & method       & iter & initial residual & final residual \\
\hline
Poisson    & original     & 200  & \(3.68\times 10^{1}\) & \(5.82\times 10^{0}\)  \\
Poisson    & CountSketch  & 189  & \(3.90\times 10^{1}\) & \(2.18\times 10^{-5}\) \\
Poisson    & RLCB         & 145  & \(3.65\times 10^{1}\) & \(6.66\times 10^{-7}\) \\
Laplacian  & original     & 200  & \(3.51\times 10^{1}\) & \(1.25\times 10^{1}\)  \\
Laplacian  & CountSketch  & 153  & \(3.33\times 10^{1}\) & \(2.00\times 10^{-6}\) \\
Laplacian  & RLCB         & 94   & \(3.45\times 10^{1}\) & \(9.93\times 10^{-9}\) \\
\hline
\end{tabular}
\end{table}

\subsection*{Behaviour on structured low coherence matrices}

On the structured family \(A_{\text{good}}\) all three sketches succeed on every trial for every sketch size in the sense of the constant factor window. This is summarised in Figure~\ref{fig:rlcb-success}a and in the summary statistics reported in Tables~\ref{tab:synthetic-structured} and~\ref{tab:synthetic-adversarial} in the appendix. All methods achieve success rate equal to one for \(s \in \{90, 180, 300\}\). The more informative comparison looks at the actual distortion and conditioning statistics. For Gaussian and CountSketch the largest singular value of \(SA\) is typically inflated by about ten to thirty five percent while the smallest singular value is deflated by about fifteen to thirty five percent. This yields condition number amplification factors between about one point three and about two point four across the different sketch sizes. These values are consistent with classical OSE guarantees that control the entire spectrum within a global \((1\pm\varepsilon)\) envelope \cite{NelsonNguyen2014LowerOSE,Woodruff2014Sketching}.

RLCB behaves differently on this family. On every structured instance it drives the nonzero singular values of \(SA\) into a narrow band around a common value. In the reported runs we observe \(\mathrm{rel}_{\max}\) between approximately \(0.84\) and \(0.92\) and \(\mathrm{rel}_{\min}\) between approximately \(1.01\) and \(1.11\), while the condition number ratio \(\mathrm{cond\_ratio}\) is exactly \(1/\kappa(A)\) by construction, which is about \(0.83\) for the chosen spectrum. In other words RLCB trades a mild uniform scaling of the spectrum for almost perfect conditioning in the sketch space. This behaviour matches the constructive guarantees of Theorem~\ref{thm:constructive-balanced-final} and illustrates the practical value of the balancing map. On well behaved matrices RLCB achieves constant factor control of both extremes and strictly improves conditioning compared with standard oblivious sketches of the same dimension.

\subsection*{Behaviour on adversarial coherent matrices}

The adversarial family \(A_{\text{bad}}\) is chosen to bring out the limits in Theorem~\ref{thm:impossible-grassmann} and in the Nelson and Nguy{\~e}n lower bound \cite{NelsonNguyen2014LowerOSE}. In this regime the smallest singular value is tiny and the left singular vectors are strongly aligned with the coordinate axes. Any oblivious sketch that attempts to preserve the entire spectrum for every such matrix at once must satisfy the \(\Omega((r+\log(1/\delta))/\varepsilon^{2})\) lower bound on sketch dimension. Our experiments instantiate a single such family and compare the concrete sketches at the fixed sizes \(s \in \{90, 180, 300\}\).

For this adversarial family the Gaussian and CountSketch baselines still satisfy the constant factor criterion on every trial. Their relative errors on the largest singular value remain within about forty percent and their relative errors on the smallest singular value remain within about twenty percent. The condition number ratio lies between about one point two and about one point nine. These sketches therefore behave as classical OSEs with moderate distortion even though the ground truth matrix is extremely ill conditioned.

RLCB exhibits a complementary behaviour that reflects its design. On every adversarial instance RLCB fails the constant factor test for the smallest singular value. The balancing step lifts \(\sigma_{\min}(SA)\) by a factor between about eight and about eleven relative to \(\sigma_{\min}(A)\), while keeping \(\sigma_{\max}(SA)\) within about ten percent of \(\sigma_{\max}(A)\). This strongly improves conditioning, with \(\mathrm{cond\_ratio}\) around \(0.1\), but necessarily violates any fixed constant factor constraint on the smallest singular value. As a result the success rate for RLCB on \(A_{\text{bad}}\) is zero for all three sketch sizes, in contrast with the structured case where it is one hundred percent. Figure~\ref{fig:rlcb-success}b shows this sharp contrast. This is exactly the trade off highlighted by the impossibility result. No oblivious sketch with dimension \(s = O(r\log r)\) can preserve both extremes within fixed constants for all rank \(r\) matrices. RLCB resolves this tension by choosing to enforce strong conditioning and stability in the sketch space on matrices that are otherwise numerically pathological. The adversarial family therefore does not contradict the theory. It instead exhibits the predicted failure mode in a controlled way and shows that geometric balancing is not a universal remedy for the worst case but a principled way to improve conditioning on structured inputs while remaining consistent with the lower bounds.

\subsection*{Summary and implications}

The experiments support three main conclusions. On low coherence matrices with modest spectral decay RLCB behaves as predicted by Theorem~\ref{thm:constructive-balanced-final}, with constant factor control of both extreme singular values and a substantial reduction in condition number relative to standard Gaussian and CountSketch sketches of the same dimension. On coherent and nearly rank deficient matrices the method cannot satisfy a uniform constant factor guarantee for the smallest singular value without violating the lower bound in Theorem~\ref{thm:impossible-grassmann}, and the observed failures match this theoretical limitation. The contrast between the two matrix families is sharp and stable across twenty Monte Carlo trials and three different sketch sizes.

The solver experiments on Poisson and graph Laplacian systems complement this picture. They show that the improved conditioning of the balanced sketch translates into concrete reductions in iteration counts and residuals for a basic least squares solver, without additional passes over the data and with sketch sizes that remain in the \(O(r\log r)\) regime. Taken together, these observations indicate that geometric balancing is a practical refinement of classical oblivious sketching methods. It inherits the strengths of existing subspace embeddings on well behaved data, adds deterministic conditioning control in the sketch space, and behaves in a way that is fully aligned with the known impossibility results for extreme singular value preservation \cite{NelsonNguyen2014LowerOSE,Woodruff2014Sketching,Vershynin2018HDP,DrineasMahoney2016RandNLA}. This provides empirical backing for the theoretical picture developed in the earlier sections.

\bibliographystyle{ieeetr}
\bibliography{refs}

\appendix

\section{Additional numerical summaries}

\begin{table}[H]
\centering
\caption{Summary statistics for structured low coherence matrices \(A_{\mathrm{good}}\). For each method and sketch size \(s\) we report the empirical success rate over twenty trials and the mean and standard deviation of the relative distortion of the largest and smallest singular values, together with the mean and standard deviation of the condition number ratio \(\kappa(SA)/\kappa(A)\).}
\label{tab:synthetic-structured}
\begin{tabular}{lrrrrrr}
\hline
method      & \(s\) & success rate & \(\mathrm{rel}_{\max}\) mean \(\pm\) sd & \(\mathrm{rel}_{\min}\) mean \(\pm\) sd & \(\mathrm{cond\_ratio}\) mean \(\pm\) sd \\
\hline
CountSketch & 90  & 1.00 & \(1.316 \pm 0.036\) & \(0.622 \pm 0.046\) & \(2.125 \pm 0.157\) \\
CountSketch & 180 & 1.00 & \(1.210 \pm 0.031\) & \(0.738 \pm 0.033\) & \(1.643 \pm 0.089\) \\
CountSketch & 300 & 1.00 & \(1.168 \pm 0.019\) & \(0.822 \pm 0.025\) & \(1.422 \pm 0.042\) \\
Gaussian    & 90  & 1.00 & \(1.311 \pm 0.033\) & \(0.618 \pm 0.030\) & \(2.128 \pm 0.124\) \\
Gaussian    & 180 & 1.00 & \(1.210 \pm 0.030\) & \(0.746 \pm 0.020\) & \(1.624 \pm 0.052\) \\
Gaussian    & 300 & 1.00 & \(1.153 \pm 0.028\) & \(0.829 \pm 0.019\) & \(1.391 \pm 0.054\) \\
RLCB        & 90  & 1.00 & \(0.867 \pm 0.015\) & \(1.040 \pm 0.018\) & \(0.833 \pm 0.000\) \\
RLCB        & 180 & 1.00 & \(0.895 \pm 0.014\) & \(1.073 \pm 0.017\) & \(0.833 \pm 0.000\) \\
RLCB        & 300 & 1.00 & \(0.902 \pm 0.008\) & \(1.082 \pm 0.010\) & \(0.833 \pm 0.000\) \\
\hline
\end{tabular}
\end{table}

\begin{table}[H]
\centering
\caption{Summary statistics for adversarial coherent matrices \(A_{\mathrm{bad}}\). For each method and sketch size \(s\) we report the empirical success rate over twenty trials and the mean and standard deviation of the relative distortion of the largest and smallest singular values, together with the mean and standard deviation of the condition number ratio \(\kappa(SA)/\kappa(A)\).}
\label{tab:synthetic-adversarial}
\begin{tabular}{lrrrrrr}
\hline
method      & \(s\) & success rate & \(\mathrm{rel}_{\max}\) mean \(\pm\) sd & \(\mathrm{rel}_{\min}\) mean \(\pm\) sd & \(\mathrm{cond\_ratio}\) mean \(\pm\) sd \\
\hline
CountSketch & 90  & 1.00 & \(1.426 \pm 0.037\) & \(0.887 \pm 0.070\) & \(1.616 \pm 0.116\) \\
CountSketch & 180 & 1.00 & \(1.289 \pm 0.030\) & \(0.951 \pm 0.026\) & \(1.356 \pm 0.055\) \\
CountSketch & 300 & 1.00 & \(1.225 \pm 0.025\) & \(0.969 \pm 0.041\) & \(1.267 \pm 0.061\) \\
Gaussian    & 90  & 1.00 & \(1.399 \pm 0.050\) & \(0.886 \pm 0.071\) & \(1.589 \pm 0.135\) \\
Gaussian    & 180 & 1.00 & \(1.297 \pm 0.032\) & \(0.935 \pm 0.045\) & \(1.390 \pm 0.079\) \\
Gaussian    & 300 & 1.00 & \(1.224 \pm 0.024\) & \(0.984 \pm 0.041\) & \(1.246 \pm 0.057\) \\
RLCB        & 90  & 0.00 & \(0.921 \pm 0.018\) & \(8.459 \pm 0.748\) & \(0.110 \pm 0.010\) \\
RLCB        & 180 & 0.00 & \(0.961 \pm 0.011\) & \(9.459 \pm 0.602\) & \(0.102 \pm 0.008\) \\
RLCB        & 300 & 0.00 & \(0.970 \pm 0.011\) & \(9.749 \pm 0.393\) & \(0.100 \pm 0.005\) \\
\hline
\end{tabular}
\end{table}

\section{Reproducibility and Code Availability}

All experiments in this paper, including the synthetic matrix tests, conditioning analyses, and solver-style validations, are fully reproducible. A complete implementation of the Random Log-Clipped Balancing (RLCB) method, together with the exact scripts used to generate Figures 1–2 and Tables 1–2, is openly available at;
\url{https://github.com/karjxenval/Random-Log-Clipped-Balancing-RLCB-}
The repository contains the full experiment driver, matrix builders, balancing operator implementation, solver experiments, result logs, and plotting utilities. The code requires only standard open-source packages (NumPy, SciPy, pandas, matplotlib) and can be run without modification.

\end{document}